\documentclass{scrartcl}
\usepackage[margin=1.1in]{geometry}

\AtEndDocument{\bigskip{%
  \textsc{Department of Mathematics, Boyd Graduate Studies Building, University of Georgia, Athens, Georgia 30602} \par
  \textit{E-mail address}: \texttt{pollack@uga.edu}
}}

\usepackage[english]{babel}
\usepackage[utf8]{inputenc}
\usepackage{url}
\usepackage{amsmath,amsthm,amssymb}
\usepackage[protrusion=true]{microtype}
\usepackage[T1]{fontenc}
\usepackage{kurier}

\theoremstyle{plain}
\newtheorem{thm}{Theorem}[section]
\newtheorem{lem}[thm]{Lemma}
\newtheorem{prop}[thm]{Proposition}
\theoremstyle{remark}
\newtheorem*{remark}{Remark}
\newtheorem*{rmk}{Remarks}
\renewcommand{\phi}{\varphi}
\newcommand{\Ss}{\mathcal{S}}

\newcommand{\Z}{\mathbf{Z}}
\newcommand{\Q}{\mathbf{Q}}
\newcommand{\sA}{\mathcal{A}}

\newcommand\leg{\genfrac(){.4pt}{}}
\title{{Bounds for the first several prime character nonresidues}}

\author{{Paul Pollack}}
\date{\vspace{-3.2ex}}

\begin{document}
\maketitle

\begin{abstract}\noindent Let $\varepsilon > 0$. We prove that there are constants $m_0=m_0(\varepsilon)$ and $\kappa=\kappa(\varepsilon) > 0$ for which the following holds: For every integer $m > m_0$ and every nontrivial Dirichlet character modulo $m$, there are more than $m^{\kappa}$ primes $\ell \le m^{\frac{1}{4\sqrt{e}}+\varepsilon}$ with $\chi(\ell)\notin \{0,1\}$. The proof uses the fundamental lemma of the sieve, Norton's refinement of the Burgess bounds, and a result of Tenenbaum on the  distribution of smooth numbers satisfying a coprimality condition. For quadratic characters, we demonstrate a somewhat weaker lower bound on the number of primes $\ell \le m^{\frac14+\epsilon}$ with $\chi(\ell)=1$.
\end{abstract}

\section{Introduction}
Let $\chi$ be a nonprincipal Dirichlet character.  An integer $n$ is called a $\chi$-\emph{nonresidue} if $\chi(n) \notin \{0,1\}$. Problems about character nonresidues go back to the beginnings of modern number theory. Indeed, one can read out of Gauss's \emph{Disquisitiones} that for primes $p\equiv 1\pmod{8}$ and $\chi(\cdot) = \leg{p}{\cdot}$, the smallest $\chi$-nonresidue does not exceed $2\sqrt{p}+1$ \cite[Article 129]{gauss86}. This was an auxiliary result required for Gauss's first proof of the quadratic reciprocity law.

In the early 20th century, I.\,M. Vinogradov initiated the study of how the quadratic residues and nonresidues modulo a prime $p$ are distributed in the interval $[1,p-1]$. A particularly natural problem is to estimate the size of $n_p$, the smallest quadratic nonresidue modulo $p$. Vinogradov conjectured that $n_p \ll_{\varepsilon} p^{\varepsilon}$, for each $\varepsilon >0$. By means of a novel estimate for character sums (independently discovered by P\'olya), coupled with a clever sieving argument, he showed \cite{vinogradov18} that $n_p \ll_{\varepsilon} p^{\frac{1}{2\sqrt{e}} + \varepsilon}$. Burgess's character sum bounds \cite{burgess57}, in conjunction with Vinogradov's methods, yield the sharper estimate
\begin{equation}\label{eq:burgessquadratic} n_p \ll_{\varepsilon} p^{\frac{1}{4\sqrt{e}}+\varepsilon}. \end{equation}
Fifty years of subsequent research has not led to any improvement in the exponent $\frac{1}{4\sqrt{e}}$. But generalizing \eqref{eq:burgessquadratic}, Norton showed that if $\chi$ is any nontrivial character modulo $m$, then the least $\chi$-nonresidue is $O_{\varepsilon}(m^{1/4\sqrt{e} + \varepsilon})$. See \cite[Theorem 1.30]{norton98}.

Since $\chi$ is completely multiplicative, the smallest $\chi$-nonresidue is necessarily prime. In this note, we prove that there are actually many prime $\chi$-nonresidues satisfying the Burgess--Norton upper bound.

\begin{thm}\label{thm:main} For each $\varepsilon > 0$, there are numbers $m_0(\varepsilon)$ and $\kappa=\kappa(\varepsilon)> 0$ for which the following holds: For all $m > m_0$ and each nontrivial character $\chi$ mod $m$, there are more than $m^{\kappa}$ prime $\chi$-nonresidues not exceeding $m^{\frac{1}{4\sqrt{e}}+\varepsilon}$.
\end{thm}

The problem of obtaining an upper bound on the first several prime character nonresidues was considered already by Vinogradov. In \cite{vinogradov18},  he showed that for large $p$, there are at least $\frac{\log{p}}{7\log\log{p}}$ prime quadratic nonresidues modulo $p$ not exceeding \[ p^{\frac{1}{2}-\frac{1}{\log\log{p}}}. \]
For characters to prime moduli, a result resembling Theorem \ref{thm:main} was proved by Hudson in 1983 \cite{hudson83}. (See also Hudson's earlier investigations \cite{hudson73,hudson74,hudson74A}.) But even restricted to prime $m$, Theorem \ref{thm:main} improves on \cite{hudson83} in multiple respects. In \cite{hudson83}, the exponent on $p$ is $\frac{1}{4}+\varepsilon$ instead of $\frac{1}{4\sqrt{e}} +\varepsilon$, and the number of nonresidues produced is only $c_{\varepsilon} \frac{\log{p}}{\log\log{p}}$. Moreover, it is assumed in \cite{hudson83} that the order of $\chi$ is fixed. Stronger results than those of \cite{hudson83} were  announced by Norton already in 1973 \cite{norton74}.\footnote{Norton claims in \cite{norton74}: \emph{Let $\varepsilon>0$ and $k_0 \ge 2$. If $m \ge 3$ and $[(\Z/m\Z)^{\times}: {(\Z/m\Z)^{\times}}^k] \ge k_0$, then each of the smallest $\lfloor \log{m}/\log\log{m}\rfloor$ primes not dividing $m$ that are $k$th power nonresidues modulo $m$ is $\ll_{\varepsilon,k_0}n^{1/4u_{k_0} + \varepsilon}$}. Here $u_{k_0}$ has the same meaning as in our introduction.}  Unfortunately, a full account of Norton's work seems to have never appeared.

It becomes easier to produce small character nonresidues as the order of $\chi$ increases. This phenomenon was noticed by Vinogradov \cite{vinogradov27} and further investigated by Buchstab \cite{buchstab49} and Davenport and Erd\H{o}s \cite{DE52}. To explain their results requires us to first recall the rudiments of the theory of smooth numbers. For each positive integer $n$, let $P^{+}(n)$ denote the largest prime factor of $n$, with the convention that $P^{+}(1)=1$. A natural number $n$ is called \emph{$y$-smooth} (or \emph{$y$-friable}) if $P^{+}(n) \le y$. For $x \ge y \ge 2$, we let $\Psi(x,y)$ be the count of $y$-smooth numbers up to $x$. We let $\rho$ be Dickman's function, defined by
\[ \rho(u)=1\text{ for $0 \le u \le 1$}, \quad \text{and}\quad u \rho'(u) = -\rho(u-1) \quad\text{for $u > 1$}. \]
The functions $\Psi(x,y)$ and $\rho(u)$ are intimately connected; it is known that $\Psi(x,y) \sim x\rho(u)$, where $u:=\frac{\log{x}}{\log{y}}$, in a wide range of $x$ and $y$. In fact, Hildebrand \cite{hildebrand86} has shown that this asymptotic formula holds whenever $x\to\infty$, as long as
\[ y \ge \exp((\log\log{x})^{5/3+\lambda}) \]
for some fixed positive $\lambda$. For this estimate to be useful, one needs to understand the behavior of $\rho(u)$. It is not hard to show that $\rho$ is strictly decreasing for $u > 1$ and that $\rho(u) \le 1/\Gamma(u+1)$. So for any $k > 1$, there is a unique $u_k > 1$ with $\rho(u_k)=\frac{1}{k}$. Buchstab and, independently, Davenport and Erd\H{o}s (developing ideas implicit in \cite{vinogradov27}) showed that if $\chi$ mod $p$ has order $k \ge 2$, then the least $\chi$-nonresidue is $O_{\varepsilon,k}(p^{1/2u_k+\varepsilon})$. If in their argument Burgess's method (which was not available at the time) is used in place of the P\'olya--Vinogradov inequality, then $1/2u_k$ may be replaced by $1/4u_k$ \cite{wy64}. We prove the following:

\begin{thm}\label{thm:fixedprime} Let $\varepsilon >0$ and $k_0 \ge 2$. There are numbers $m_0(\varepsilon,k_0)$ and $\kappa = \kappa(\varepsilon,k_0) > 0$ for which the following holds: For all $m > m_0$ and each nontrivial character $\chi$ mod $m$ of order $k \ge k_0$, there are more than $m^{\kappa}$ prime $\chi$-nonresidues not exceeding $m^{\frac{1}{4u_{k_0}}+\varepsilon}$.\end{thm}
\begin{rmk}\mbox{ }
\begin{itemize}
\item It follows readily from the definition that $\rho(u) = 1-\log{u}$ for $1 \le u \le 2$,  and so $u_2 =  e^{1/2} = 1.6487\ldots$ and $u_3 = e^{2/3} = 1.9477\ldots$. For $k > 3$, it does not seem that $u_k$ has a simple closed form expression.
\item Theorem \ref{thm:main} is the special case $k_0=2$ of Theorem \ref{thm:fixedprime}.
\end{itemize}
\end{rmk}

One might compare Theorem \ref{thm:main} for the quadratic character modulo a prime $p$ with a result of Banks--Garaev--Heath-Brown--Shparlinski \cite{BGHBS08}. They show that for each fixed $\varepsilon > 0$, and each $N \ge p^{1/4\sqrt{e}+\varepsilon}$, the proportion of quadratic nonresidues modulo $p$ in $[1,N]$ is $\gg_{\varepsilon} 1$ for all primes $p > p_0(\varepsilon)$.

Our arguments use the ideas of Vinogradov and Davenport--Erd\H{o}s but take advantage of modern developments in sieve methods and the theory of smooth numbers. A variant of the Burgess bounds developed by Norton also plays an important role. We note that an application of the sieve that is similar in spirit to ours appears in work of Bourgain and Lindenstrauss \cite[Theorem 5.1]{BL03}.\footnote{A special case of their result: \emph{Given $\varepsilon >0$, there is an $\alpha>0$ such that $\sum_{\substack{p^{\alpha} \le \ell \le p^{1/4+\varepsilon} \\ \leg{\ell}{p}=-1}}\frac{1}{\ell} > \frac{1}{2}-\varepsilon$, for all $p > p_0(\varepsilon)$.}}

It is equally natural to ask for small prime character \emph{residues}, i.e., primes $\ell$ with $\chi(\ell)=1$. The most significant unconditional result in this direction is due to Linnik and A.\,I. Vinogradov \cite{VL66}. They showed that if $\chi$ is the quadratic character modulo a prime $p$, then the smallest prime $\ell$ with $\chi(\ell)=1$ satisfies $\ell \ll_{\varepsilon} p^{1/4+\varepsilon}$. More generally, Elliott \cite{elliott71} proved that when $\chi$ has order $k$, the least such $\ell$ is $O_{k,\varepsilon}(p^{\frac{k-1}{4}+\epsilon})$. As Elliott notes, this bound is only interesting for small values of $k$; otherwise, it is inferior to what follows from known forms of Linnik's theorem on primes in progressions. For extensions of the Linnik--Vinogradov method in a different direction, see \cite{pollack14B, pollack14}.

Our final result is a partial analogue of Theorem \ref{thm:main} for prime residues of quadratic characters. Regrettably, the number of primes produced falls short of a fixed power of $m$.

\begin{thm}\label{thm:smallresidue} Let $\varepsilon > 0$ and let $A >0$. There is an $m_0=m_0(\varepsilon,A)$ with the following property: If $m > m_0$, and $\chi$ is a quadratic character modulo $m$, then there are at least $(\log{m})^{A}$ primes $\ell \le m^{\frac{1}{4}+\varepsilon}$ with $\chi(\ell)=1$.
\end{thm}
Results of the sort proven here have direct consequences for prime splitting in cyclic extensions of $\Q$. For example, Theorem \ref{thm:main} (respectively Theorem \ref{thm:smallresidue}) implies that there are more than $|\Delta|^{\kappa}$ inert (respectively, more than $(\log|\Delta|)^{A}$ split) primes $p \le |\Delta|^{\frac{1}{4\sqrt{e}}+\varepsilon}$ (respectively, $p \le |\Delta|^{\frac{1}{4}+\varepsilon}$) in the quadratic field of discriminant $\Delta$, as soon as $|\Delta|$ is large enough in terms of $\varepsilon$ (and $A$).

\section{Small prime nonresidues: Proofs of Theorems \ref{thm:main} and \ref{thm:fixedprime}}

\subsection{Preparation} As might be expected, the Burgess bounds play the key role in our analysis. The following version is due to Norton (see \cite[Theorem 1.6]{norton98}).

\begin{prop}\label{prop:norton} Let $\chi$ be a nontrivial character modulo $m$ of order dividing $k$. Let $r$ be a positive integer,
and let $\epsilon > 0$. For all $x > 0$,
\[ \sum_{n \le x} \chi(n) \ll_{\epsilon,r} R_k(m)^{1/r} x^{1-\frac{1}{r}} m^{\frac{r+1}{4r^2}+\epsilon}. \]
Here
\[ R_k(m) = \min\left\{M(m)^{3/4},Q(k)^{9/8}\right\}, \]
where
\[ M(m) = \prod_{p^e \parallel m,~e\ge 3} p^e \qquad\text{and}\quad Q(k) = \prod_{p^e \parallel k,~e\ge 2} p^e. \]
The factor of $R_k(m)^{1/r}$ can be omitted if $r \le 3$.
\end{prop}

Another crucial tool is a theorem of Tenenbaum concerning the distribution of smooth numbers satisfying a coprimality condition. For $x\ge y\ge 2$, let \[ \Psi_q(x,y) = \#\{n\le x: \gcd(n,q)=1, P^{+}(q) \le y\}. \]

\begin{prop}\label{prop:FT} For positive integers $q$ and real numbers $x, y$ satisfying
\[ P^{+}(q) \le y \le x \quad\text{and}\quad \omega(q) \le y^{1/\log(1+u)}, \]
we have
\[ \Psi_q(x,y) = \frac{\phi(q)}{q} \Psi(x,y) \left(1+O\left(\frac{\log(1+u) \log(1+\omega(q))}{\log{y}}\right)\right). \]
As before, $u$ denotes the ratio $\log{x}/\log{y}$.
\end{prop}
\begin{proof} This is the main result of \cite{tenenbaum93} in the case $A=1$.\end{proof}

\begin{remark} If $q'$ is the largest divisor of $q$ supported on the primes not exceeding $y$, then $\Psi_{q}(x,y) = \Psi_{q'}(x,y)$. So the assumption in Proposition \ref{prop:FT} that $P^{+}(q) \le y$ does not entail any loss of generality.
\end{remark}

Theorem \ref{thm:fixedprime} will be deduced from two variant results claiming weaker upper bounds.

\begin{thm}\label{thm:fixedprime0} Let $\varepsilon >0$ and $k_0 \ge 2$. There are numbers $m_0(\varepsilon,k_0)$ and $\kappa = \kappa(\varepsilon,k_0) > 0$ for which the following holds: For all $m > m_0$ and each nontrivial character $\chi$ mod $m$ of order $k \ge k_0$, there are more than $m^{\kappa}$ prime $\chi$-nonresidues not exceeding $m^{\frac{1}{3u_{k_0}}+\varepsilon}$.\end{thm}

\begin{thm}\label{thm:fixedprime1} Let $\varepsilon >0$ and $k_0 \ge 2$. There are numbers $m_0(\varepsilon,k_0)$ and $\kappa = \kappa(\varepsilon,k_0) > 0$ for which the following holds: For all $m > m_0$ and each nontrivial character $\chi$ mod $m$ of order $k \ge k_0$, there are more than $m^{\kappa}$ prime $\chi$-nonresidues not exceeding $R_k(m) m^{\frac{1}{4u_{k_0}}+\varepsilon}$. Here $R_k(m)$ is as defined in Proposition \ref{prop:norton}.
\end{thm}

The proof of Theorem \ref{thm:fixedprime1} is given in detail in the next section. We include only a brief remark about the proof of Theorem \ref{thm:fixedprime0}, which is almost entirely analogous (but slightly simpler). We then present the derivation of Theorem \ref{thm:fixedprime} from Theorems \ref{thm:fixedprime0} and \ref{thm:fixedprime1}. We remind the reader that Theorem \ref{thm:main} is the special case $k_0=2$ of Theorem \ref{thm:fixedprime}.

\subsection{Proof of Theorem \ref{thm:fixedprime1}}\label{sec:proofs} We let $\chi$ be a nontrivial character modulo $m$ of order $k \ge k_0$, where $k_0 \ge 2$ is fixed. With $\delta \in (0,\frac{1}{4})$, we set
\[ x= R_k(m) \cdot m^{\frac14 + \delta}, \quad y = x^{\frac{1}{u_{k_0}} + \delta}. \]
To prove Theorem \ref{thm:fixedprime1}, it suffices to show that for all large $m$ (depending only on $k_0$ and $\delta$), there are at least $x^{\kappa}$ prime $\chi$-nonresidues in $[1,y]$ for a certain constant $\kappa = \kappa(k_0,\delta) > 0$.

Let $q$ be the product of the prime $\chi$-nonresidues in $[1,y]$. Note that $\gcd(q,m)=1$, from the definition of a $\chi$-nonresidue.
Our strategy is to estimate
\begin{equation}\label{eq:different} \sum_{\substack{n \le x \\ \gcd(n,mq)=1}} (1+\chi(n) + \chi^2(n) + \dots + \chi^{k-1}(n)) \end{equation}
in two different ways.

We first derive a lower bound on \eqref{eq:different}, under the assumption that there are not so many prime $\chi$-nonresidues in $[1,y]$.

\begin{lem}\label{lem:lower} There are constants $\eta = \eta(\delta,k_0) > 0$, $\kappa =\kappa(\delta,k_0) > 0$, and $m_0 = m_0(\delta,k_0)$ with the following property: If $m > m_0$ and $\omega(q) \le x^{\kappa}$, then
\[ \sum_{\substack{n \le x \\ \gcd(n,mq)=1}} (1+\chi(n) + \dots +\chi(n)^{k-1}) \ge \left(1+\frac{2k}{3}\eta\right) \frac{\phi(mq)}{mq} x. \]
\end{lem}
\begin{proof} Observe that
\[ \sum_{\substack{n \le x \\ \gcd(n,mq)=1}} (1+\chi(n) + \dots + \chi(n)^{k-1}) = k \sum_{\substack{n \le x \\ \gcd(n,q)=1,~\chi(n)=1}} 1  \ge k\sum_{\substack{n \le x\\ \gcd(n,mq)=1 \\ p \mid n \Rightarrow p \le y}} 1 = k \cdot \Psi_{mq}(x,y).\]
We estimate $\Psi_{mq}(x,y)$ using Proposition \ref{prop:FT} and the succeeding remark. We have $u \asymp_{k_0} 1$, or equivalently, $\log y\asymp_{k_0} \log{x}$. So if $\kappa$ is sufficiently small in terms of $k_0$, and $\omega(q) \le x^{\kappa}$, Proposition \ref{prop:FT} gives
\begin{align*} \Psi_{mq}(x,y) &= \bigg(\Psi(x,y)\prod_{\substack{p \mid mq \\ p \le y}}\left(1-\frac{1}{p}\right)\bigg) \left(1+O_{k_0}\left(\frac{\log(1+x^{\kappa})}{\log{x}}\right)\right)\\
&\ge \Psi(x,y) \frac{\phi(mq)}{mq}  \left(1+O_{k_0}\left(\frac{\log(1+x^{\kappa})}{\log{x}}\right)\right).\end{align*}
Now the result of Hildebrand quoted in the introduction (or a much more elementary theorem) shows that $\Psi(x,y) = \Psi(x, x^{\frac{1}{u_{k_0}}+\delta}) \ge (\frac{1}{k_0}+\eta) x$ for a certain $\eta= \eta(k_0,\delta) > 0$ and all large $x$. So if $\kappa$ is fixed sufficiently small, depending on $k_0$ and $\delta$, and $x$ is sufficiently large,
\[ \Psi_{mq}(x,y) > \left(\frac{1}{k_0} + \frac{2}{3}\eta\right) \frac{\phi(mq)}{mq} x. \]
Hence,
\[ \sum_{\substack{n \le x \\ \gcd(n,q)=1}} (1+\chi(n) + \dots + \chi(n)^{k-1}) \ge \left(\frac{k}{k_0} + \frac{2k}{3}\eta\right)\frac{\phi(mq)}{mq} x \ge \left(1+\frac{2k}{3}\eta\right) \frac{\phi(mq)}{mq} x. \qedhere\]
\end{proof}

We turn next to an upper bound.

\begin{lem}\label{lem:upper} Let $\beta > 0$. There are numbers $\eta' = \eta'(\delta) > 0$, $\kappa' = \kappa'(\delta,\beta)>0$ and $m_0 = m_0(\delta,\beta)$ with the following property: If $m > m_0$ and $\omega(q) \le x^{\kappa'}$, then
\[ \sum_{\substack{n \le x \\ \gcd(n,mq)=1}} (1+\chi(n) + \chi(n)^2 + \dots + \chi(n)^{k-1}) \le (1+\beta) \frac{\phi(mq)}{mq}x +O_{\delta}(k x^{1-\eta'}). \]
\end{lem}
\begin{proof} We let $\sA = \{n \le x: \gcd(n,m)=1,~\chi(n)=1\}$ and observe that
\begin{equation}\label{eq:fundidentity} \sum_{\substack{n \le x \\ \gcd(n,mq)=1}} (1+\chi(n) + \chi(n)^2 + \dots + \chi(n)^{k-1}) = k \sum_{\substack{n \in \sA \\ \gcd(n,q)=1}} 1. \end{equation}
We apply the fundamental lemma of the sieve to estimate the right-hand sum. (The precise form of the fundamental lemma is not so important, but we have in mind \cite[Theorem 4.1, p. 29]{diamond08}.) Let $d \in [1,x]$ be a squarefree integer dividing $q$. Then
\begin{align*} \sum_{\substack{n \in \sA \\ d \mid n}} 1 &= \frac{1}{k} \sum_{\substack{n \le x \\\gcd(n,m)=1,~d\mid n}} (1+\chi(n) + \dots + \chi(n)^{k-1}).
\end{align*}
For each $j=0,1,2,\dots, k-1$,
\[ \sum_{\substack{n \le x \\ \gcd(n,m)=1,~d\mid n}}\chi^j(n) = \chi^j(d) \sum_{\substack{e \le x/d \\ \gcd(e,m)=1}} \chi^j(e). \]
When $j=0$, the right-hand side is $\frac{x}{d}\frac{\phi(m)}{m} +O_{\epsilon}(m^{\epsilon})$, by a straightforward inclusion-exclusion. For $j \in \{1,2,\dots, k-1\}$,  Proposition \ref{prop:norton} gives
\begin{align*} \sum_{\substack{e \le x/d \\ \gcd(e,m)=1}} \chi^j(e) = \sum_{e \le x/d} \chi^j(e) \sum_{\substack{f \mid e \\ f\mid m}} \mu(f) &= \sum_{f \mid m} \mu(f) \chi^j(f) \sum_{g \le x/df} \chi^j(g) \\
&\ll_{\epsilon,r} R_k(m)^{1/r} x^{1-\frac1r} d^{-1+\frac{1}{r}} m^{\frac{r+1}{4r^2}+\epsilon} \sum_{f \mid m} f^{-1+\frac{1}{r}} \\&\ll_{\epsilon} R_k(m)^{1/r} x^{1-\frac1r} d^{-1+\frac{1}{r}} m^{\frac{r+1}{4r^2}+2\epsilon};
\end{align*}
here $r \ge 2$ and $\epsilon > 0$ are parameters to be chosen. (We used in the last step that the sum on $f$ has only $O_{\epsilon}(m^{\epsilon})$ terms, each of which is $O(1)$.) Assembling the preceding estimates,
\[ \sum_{\substack{n \in \sA \\ d \mid n}} 1 = \frac{x}{dk}\frac{\phi(m)}{m} + r(d),
\quad\text{where}\quad r(d) \ll_{\epsilon,r} R_k(m)^{1/r} x^{1-\frac1r} d^{-1+\frac{1}{r}} m^{\frac{r+1}{4r^2}+2\epsilon}. \]
By the fundamental lemma, for any choices of real parameters $z\ge 2$ and $v\ge 1$ with $z^{2v} < x$,
\begin{multline*} \sum_{\substack{n \in \sA\\\gcd(n,q)=1}} 1 \le \sum_{\substack{n \in \sA \\ p \mid \gcd(n,q) \Rightarrow p \ge z}} 1 = \Bigg(\frac{x}{k}\frac{\phi(m)}{m} \prod_{\substack{p \mid q \\ p< z}}\left(1-\frac{1}{p}\right)\Bigg)\left(1 + O(v^{-v})\right) \\ + O_{\epsilon,r}\Bigg(R_k(m)^{1/r} x^{1-\frac1r} m^{\frac{r+1}{4r^2}+2\epsilon} \sum_{\substack{d < z^{2v} \\ d\mid q}} \mu^2(d) 3^{\omega(d)} d^{-1+\frac1r}\Bigg).\end{multline*}

We now make a choice of parameters. Let $r = \lceil \frac{1}{2\delta}\rceil$ (so that $\delta \ge \frac{1}{2r}$). Since $x=R_k(m)\cdot m^{1/4+\delta}$, we have
\[ R_k(m)^{1/r} x^{1-\frac1r} m^{\frac{r+1}{4r^2}} = x \cdot m^{-\frac{1}{4r} -\delta/r} m^{\frac{r+1}{4r^2}} = x \cdot m^{\frac{1}{r}(\frac{1}{4r}-\delta)} \le x\cdot m^{-\frac{\delta}{4r^2}}. \]
We take $\epsilon = \frac{\delta}{16r^2}$, so that
\[ m^{2\epsilon} = m^{\frac{\delta}{8r^2}}. \]
Since $r\ge 2$ and $3^{\omega(d)} \ll d^{1/2}$, each term in the sum on $d$ is $O(1)$. Putting it all together, the $O$-term above is
\[ \ll_{\delta} x \cdot m^{-\frac{\delta}{4r^2}} \cdot  m^{\frac{\delta}{8r^2}} \cdot z^{2v}. \]
Since $x = R_k(m) \cdot m^{1/4+\delta} \le m^{3/4} \cdot m^{1/4+\delta} < m^2$, this upper bound is $\ll_{\delta} x^{1-\frac{\delta}{16r^2}} z^{2v}$. Taking $z = x^{\frac{\delta}{64r^2 v}}$ gives a final upper bound on the $O$-term of
\[ \ll_{\delta} x^{1-\eta'},\quad\text{where}\quad \eta' = \frac{\delta}{32r^2}. \]

Turning attention to the main term, we fix $v$ large enough that the factor $1+O(v^{-v})$ is smaller than $1+\frac{1}{2}\beta$. Then our main term above does not exceed
\begin{align*}  \frac{x}{k} \frac{\phi(mq)}{mq} \left(1+\frac{1}{2}\beta\right) \prod_{\substack{p \mid q\\ p \ge z}} \left(1-\frac{1}{p}\right)^{-1} &\le \frac{x}{k} \frac{\phi(mq)}{mq} \left(1+\frac{1}{2}\beta\right) \exp\bigg(2\sum_{\substack{p \mid q \\ p \ge z}}\frac{1}{p}\bigg) \\&\le \frac{x}{k} \frac{\phi(mq)}{mq} \left(1+\frac{1}{2}\beta\right) \exp(2\omega(q) z^{-1}). \end{align*}
Take $\kappa' = \frac{\delta}{128r^2 v}$. Under the assumption that $\omega(q) \le x^{\kappa'}$, we have $2 \omega(q) z^{-1} \le 2x^{-\delta/128r^2 v}$, and $\exp(2\omega(q) z^{-1}) = 1 + O(x^{-\delta/128r^2 v})$. So once $x$ (or equivalently, $m$) is large enough, our main term is smaller than $\frac{x}{k}\frac{\phi(mq)}{mq}(1+\beta)$. So we have shown that for large $m$,
\[ \sum_{\substack{n \in \sA \\ \gcd(n,q)=1}} 1 \le \frac{x}{k}\frac{\phi(mq)}{mq}(1+\beta) + O_{\delta}(x^{1-\eta'}). \]
Recalling \eqref{eq:fundidentity} finishes the proof.
\end{proof}

\begin{proof}[Completion of the proof of Theorem \ref{thm:fixedprime}] We keep the notation from earlier in this section. Let $\eta$, $\kappa$ be as specified in Lemma \ref{lem:lower}. With $\beta = \eta/2$, choose $\eta'$ and $\kappa'$ as in Lemma \ref{lem:upper}. If $m$ is large and we assume that
\[ \omega(q) \le x^{\kappa''}, \quad\text{where} \quad\kappa'' = \min\{\kappa,\kappa'\}, \]
then these lemmas imply that
\[ \bigg(1+\frac{2k}{3}\eta\bigg) \frac{\phi(mq)}{mq} x \le \bigg(1+\frac{1}{2}\eta\bigg) \frac{\phi(mq)}{mq} x + O_{\delta}(kx^{1-\eta'}). \]
Rearranging,
\[ k \eta \frac{\phi(mq)}{mq} x \ll \frac{4k-3}{6} \eta \cdot \frac{\phi(mq)}{mq}x \ll_{\delta} k x^{1-\eta'}, \]
and so
\[ \frac{mq}{\phi(mq)} \gg_{k_0,\delta} x^{\eta'}. \]
Noting that $m < x^{4}$ and $q \le y^{\omega(q)}\le x^{\omega(q)}$, we see that for large $x$,
\[ \frac{mq}{\phi(mq)} \ll \log\log(mq+2) \ll \log\log{x} + \log(\omega(q)+2) \ll \log{x}. \]
Comparing with the above lower bound, we see that $x$, and hence $m$, is bounded. Turning it around, for $m$ large enough, there are at least $x^{\kappa''}$ prime $\chi$-nonresidues in $[1,y]$.
\end{proof}

\begin{proof}[Sketch of the proof of Theorem \ref{thm:fixedprime0}] The proof of Theorem \ref{thm:fixedprime0} is quite similar, except that now we take $x = m^{1/3+\delta}$. With this choice of $x$, we can apply the Burgess bounds with $r=3$, which allows us to omit the factor of $R_k(m)$ in the resulting estimates.\end{proof}

\subsection{Deduction of Theorem \ref{thm:fixedprime}} Let $\varepsilon> 0$ and $k_0 \ge 2$ be fixed. Let $\chi$ be a nonprincipal character mod $m$ of order $k$, where $k \ge k_0$. We would like to show that as long as $m$ is large enough there must be at least $m^{\kappa}$ prime $\chi$-nonresidues not exceeding $x^{1/4u_{k_0}+\varepsilon}$, for a certain $\kappa = \kappa(\varepsilon,k_0) > 0$. Let $k_1$ be the smallest positive integer with $3u_{k_1} > 4u_{k_0}$.  If $k \ge k_1$, apply Theorem \ref{thm:fixedprime0}: We find that for large $m$, there are at least $m^{\kappa_0}$ prime $\chi$-nonresidues
\[ \le m^{\frac{1}{3u_{k_1}} + \varepsilon} \le m^{\frac{1}{4u_{k_0}} + \varepsilon}, \]
where $\kappa_0 = \kappa(\varepsilon,k_1)$ in the notation of Theorem \ref{thm:fixedprime0}. Suppose instead that $k_0 \le k< k_1$. Then $R_k(m)$ is bounded in terms of $k_0$. Theorem \ref{thm:fixedprime1} thus shows that for large $m$, there are at least $m^{\kappa_1}$ prime $\chi$-nonresidues
\[ \le R_k(m) m^{\frac{1}{4u_{k_0}} + \varepsilon/2} \le m^{\frac{1}{4u_{k_0}}+\varepsilon}, \]
where $\kappa_1 = \kappa(\varepsilon/2,k_0)$ in the notation of Theorem \ref{thm:fixedprime1}. Theorem \ref{thm:fixedprime} follows with $\kappa = \min\{\kappa_0,\kappa_1\}$.

\begin{remark} By a minor modification of our proof, one can establish the following more general result. Theorem \ref{thm:fixedprime} corresponds to the case $H = \ker\chi$.

\begin{thm} Let $\varepsilon >0$ and $k_0 \ge 2$. There are numbers $m_0(\varepsilon,k_0)$ and $\kappa = \kappa(\varepsilon,k_0) > 0$ for which the following holds: For all $m > m_0$ and every proper subgroup $H$ of $G=(\Z/m\Z)^{\times}$ of index $k \ge k_0$, there are more than $m^{\kappa}$ primes $\ell$ not exceeding $m^{\frac{1}{4u_{k_0}}+\varepsilon}$ with $\ell \nmid m$ and $\ell \bmod{m}\notin H$.\end{thm}

\noindent This strengthens \cite[Theorem 1.20]{norton98}, where the bound $O_{k_0,\epsilon}(m^{\frac{1}{4u_{k_0}}+\varepsilon})$ is established for the first such prime $\ell$.

The main idea in the proof of the generalization is to replace $1+\chi(n)+\dots + \chi(n)^{k-1}$ with $\sum_{\chi\in \widehat{G/H}} \chi(n)$, where $\widehat{G/H}$ denotes the group of characters $\chi$ mod $m$ with $\ker \chi \supset H$. We leave the remaining details to the reader.
\end{remark}

\section{Small prime residues of quadratic characters: Proof of Theorem \ref{thm:smallresidue}}
The next proposition is a variant of \cite[Theorem 2]{VL66}. Given a character $\chi$, we let $r_{\chi}(n) = \sum_{d \mid n} \chi(d)$. Since $\chi$ will be clear from context, we will suppress the subscript.
\begin{prop}\label{prop:LV} For each $\epsilon > 0$, there is a constant $\eta = \eta(\epsilon) >0$ for which the following holds: If $\chi$ is a quadratic character modulo $m$ and $x \ge m^{1/4+\epsilon}$, then
\[ \sum_{n \le x} r(n) = L(1,\chi) x + O_{\epsilon}(x^{1-\eta}). \]                                                                                                         \end{prop}
\begin{proof} With $\upsilon = \frac{1/4+\epsilon/2}{1/4+\epsilon}$, put $y = x^{\upsilon}$, so that $y \ge m^{\frac{1}{4}+\frac{1}{2}\epsilon}$. Put $z=x/y$. By Dirichlet's hyperbola method,
\begin{equation}\label{eq:hyperbola} \sum_{n \le x} r(n)  = \sum_{d \le y}\chi(d) \sum_{e \le x/d} 1 + \sum_{e \le z} \sum_{d \le x/e} \chi(d) - \sum_{d\le y}\chi(d)\sum_{e\le z}1. \end{equation}
By Proposition \ref{prop:norton} (with $k=2$, so that $R_k(m)^{1/r}=1$), there is an $\eta_0 =\eta_0(\epsilon) > 0$ with $\sum_{d \le T} \chi(d) \ll_{\epsilon} T^{1-\eta_0} \quad\text{for all}\quad T \ge y$. Thus, the second double sum on the right of \eqref{eq:hyperbola} is $\ll_{\delta} x^{1-\eta_0} \sum_{e \le z} e^{\eta_0-1} \ll_{\delta} x (z/x)^{\eta_0} = x y^{-\eta_0}$. Similarly, the third double sum is $\ll_{\epsilon} z y^{1-\eta_0} = x y^{-\eta_0}$. Finally,
\[ \sum_{d \le y}\chi(d) \sum_{e \le x/d} 1=\sum_{d\le y} \chi(d) \left(\frac{x}{d}+O(1)\right) = xL(1,\chi) - x\sum_{d > y} \frac{\chi(d)}{d} + O(y) = x L(1,\chi) + O_{\epsilon}(xy^{-\eta_0}) + O(y). \]
(Here the sum on $d>y$ has been handled by partial summation.)
Collecting our estimates and keeping in mind that $y=x^\upsilon$, we obtain the theorem with $\eta$ defined by $1-\eta = \max\{\upsilon,1-v\eta_0\}$.
\end{proof}

\begin{proof}[Proof of Theorem \ref{thm:smallresidue}] Let $\varepsilon \in (0, \frac14)$ and let $\chi$ be a quadratic character modulo $m$. Let
\[ x = m^{\frac{1}{4}+\varepsilon}, \]
and let $q$ be the product of the primes $\ell \le x$ with $\chi(\ell)=1$. We suppose that $\omega(q) \le (\log{m})^{A}$, and we show this implies that $m$ is bounded by a constant depending on $\varepsilon$ and $A$. Throughout this proof, we suppress any dependence on $\varepsilon$ and $A$ in our $O$-notation.

By Proposition \ref{prop:LV},
\begin{equation} \sum_{n \le x} r(n) = L(1,\chi) \cdot x + O(x^{1-\eta}). \label{eq:ldub}\end{equation}
We can estimate the sum in a second way. Observe that
\begin{equation}\label{eq:rnexpression} r(n) = \prod_{\ell^e \parallel n} \left(1+\chi(\ell) + \dots + \chi(\ell^e)\right) \ge 0. \end{equation}
Hence, if the subset $\Ss$ of $[1,x]$ is chosen to contain the support of $r(n)$ on $[1,x]$, then
\[ 0 \le \sum_{n \le x} r(n) \le \#\Ss \cdot \left(\max_{n \in \Ss} r(n)\right). \]
Examining the expression in \eqref{eq:rnexpression} for $r(n)$, we see $\Ss$ can be chosen as the set of $n\le x$ where every prime that appears to the first power in the factorization of $n$ divides $mq$. For each $n \in \Ss$, we can write $n=n_1 n_2$, where $n_1$ is a squarefree divisor of $mq$ and $n_2$ is squarefull. The number of elements of $\Ss$ with $n_2 > x^{1/2}$ is $O(x^{3/4})$. For the remaining elements of $\Ss$, we have $n_1 \le x/n_2$ and $n_1$ is a squarefree product of primes dividing $mq$. There is a bijection
\[ \iota\colon \{\text{squarefree divisors of $mq$}\} \to \{\text{squarefrees composed of the first $\omega(mq)$ primes}\}
\]
with $\iota(r) \le r$ for all $r$. Hence, given $n_2$, the number of choices for $n_1$ is at most the number of integers in $[1,x/n_2]$ supported on the product of the first $\omega(mq)$ primes. By our assumption on $\omega(q)$, those primes all belong to the interval $[1, (\log{x})^{A+1}]$, once $x$ is large. Hence, given $n_2$, the number of possible values of $n_1$ is at most
\[ \Psi(x/n_2, (\log{x})^{A+1}). \]
For fixed $\theta \ge 1$, a classical theorem of de Bruijn \cite{dB66} asserts that  $\Psi(X,(\log{X})^{\theta}) = X^{1-\frac{1}{\theta}+o(1)}$, as $X\to\infty$.  Since $x/n_2 \ge x^{1/2}$, we deduce that
\[  \Psi(x/n_2, (\log{x})^{A+1}) \le (x/n_2)^{1-\frac{1}{A+2}} \]
if $x$ is large. Summing on squarefull $n_2 \le x^{1/4}$, we see that the number of elements of $\Ss$ arising in this way is $O(x^{1-\frac{1}{A+2}})$. Hence,
\[ \#\Ss \ll x^{3/4} + x^{1-\frac{1}{A+2}} \ll x^{1-\eta'}, \quad\text{where}\quad \eta'=\min\left\{\frac14,\frac{1}{A+2}\right\}. \]
Since $r(n) \le \tau(n) \ll x^{\eta'/2}$ for $n \le x$,
\begin{equation}\label{eq:udub} \sum_{n\le x} r(n) \ll \#\Ss\cdot x^{\eta'/2} \ll x^{1-\eta'/2}. \end{equation}
Comparing \eqref{eq:ldub} and \eqref{eq:udub} gives
\[ L(1,\chi) \ll x^{-\min\{\eta'/2,\eta\}}. \]
But for large $x$, this contradicts Siegel's theorem \cite[Theorem 11.14, p. 372]{MV07}.
\end{proof}

\begin{remark} Any improvement on Siegel's lower bound for $L(1,\chi)$ would boost the number of $\ell$ produced in Theorem \ref{thm:smallresidue}. Substantial improvements of this kind would have other closely related implications. For example, a simple modification of an argument of Wolke \cite{wolke69} shows that for any quadratic character $\chi$ mod $m$,
\[ \sum_{\substack{\ell \le m \\ \chi(\ell)= 1}}\frac{1}{\ell} \ge \frac{1}{2} \log\left(\frac{\varphi(m)}{m} L(1,\chi) \log{m}\right) + O(1), \]
where the $O(1)$ constant is absolute. {(Here is the short proof: By Proposition \ref{prop:LV}, $\frac{1}{m}\sum_{n \le m}r(n) \gg L(1,\chi)$. On the other hand, \cite[Theorem 5, p. 308]{tenenbaum95} yields $\frac{1}{m}\sum_{n \le m} r(n) \ll \frac{1}{\log{m}} \sum_{n \le m} \frac{r(n)}{n} \ll \frac{1}{\log{m}} \cdot \frac{m}{\phi(m)} \cdot \exp\left(2 \sum_{\ell \le m,~\chi(\ell)=1}\frac{1}{\ell}\right)$.)}
\end{remark}

\section*{Acknowledgments}  This work was motivated in part by observations made on \texttt{mathoverflow} by ``GH from MO'' \cite{52393}. The author is also grateful to ``Lucia'' for pointing out there the work of Bourgain--Lindenstrauss. He thanks Enrique Trevi\~no for useful feedback on an early draft. This research was supported by NSF award DMS-1402268.

{\small 
\providecommand{\bysame}{\leavevmode\hbox to3em{\hrulefill}\thinspace}
\providecommand{\MR}{\relax\ifhmode\unskip\space\fi MR }
\providecommand{\MRhref}[2]{%
  \href{http://www.ams.org/mathscinet-getitem?mr=#1}{#2}
}
\providecommand{\href}[2]{#2}

}
\end{document}